\documentclass{article}
\usepackage{amsmath,amssymb,amscd, float, verbatim, latexsym,subfig,cite,wrapfig,enumerate,setspace,mathtools}
\usepackage[vmargin=1.5in,hmargin=1.25in]{geometry}
\usepackage{amsthm,hyperref,amstext}
\usepackage[all]{xy}
\usepackage{lscape}
\usepackage{color}
\usepackage{nicefrac}

\newtheorem{theorem}{Theorem}[section]
\newtheorem{lemma}[theorem]{Lemma}
\newtheorem{proposition}[theorem]{Proposition}
\newtheorem{corollary}[theorem]{Corollary}

\newtheorem{definition}[theorem]{Definition}

\newtheorem{example}[theorem]{Example}

\newcommand{\eps}{\epsilon}

\newcommand{\R}{\mathbb{R}}
\newcommand{\C}{\mathbb{C}}
\newcommand{\N}{\mathbb{N}}

\newcommand{\Z}{\mathbb{Z}}
\newcommand{\T}{\mathbb{T}}

\newcommand{\cTrunc}{\hat{c}}

\newcommand{\cO}{\mathcal{O}}

\newcommand{\cX}{\mathcal{X}}

\title{
Quasiperiodicity and blowup in integrable subsystems of nonconservative nonlinear Schr\"odinger equations}

\usepackage{authblk}
\author{	Jonathan~Jaquette
}
\affil{\small
	Department of Mathematics and Statistics, Boston University,
	Boston, MA 02215, USA. \newline 
	\texttt{jaquette@bu.edu}  
}

\begin{document}

\maketitle

\begin{abstract} 
	In this paper, we study the dynamics of a class of nonlinear Schr\"odinger equation $	i u_t = \triangle u + u^p $ for $ x \in \mathbb{T}^d$. 
	We prove that the PDE is   integrable on the space of non-negative Fourier coefficients, in particular that  each Fourier coefficient of a solution can be explicitly solved by quadrature. 
	Within this subspace we demonstrate a large class of (quasi)periodic solutions all with the same frequency, as well as solutions which blowup in finite time in the $L^2$ norm.  
\end{abstract}

\bigskip


\bigskip
\bigskip
\centerline{{\bf AMS subject classifications}}
\medskip
\centerline{
	35B10,  %	Periodic solutions to PDEs
	35B44, %  	Blow-up in context of PDEs
	35Q55,  	%   NLS equations (nonlinear Schrödinger equations) {For dynamical systems and ergodic theory, see 37K10}
	37K10 % Completely integrable infinite-dimensional Hamiltonian and Lagrangian systems, integration methods, integrability tests, integrable hierarchies (KdV, KP, Toda, etc.)

}

\section{Introduction}

In this paper we consider the dynamical behavior and  global wellposedness of solutions to the spatially periodic nonlinear Schr\"odinger equation 
\begin{align} \label{eq:NLS}
\begin{dcases}
i u_t = \triangle u + u^p,  & t \in \R,\, x \in \T^d\\
u(0) = u_0, &  x \in \T^d.
\end{dcases}
\end{align} 
Nonlinear Schr\"odinger equations (NLS) have been extensively studied by mathematicians and physicists, and broadly serve as reduced models for nonlinear dispersive equations. 
The particular equation in \eqref{eq:NLS} however falls outside the class of NLS most commonly considered:  
the nonlinearity here is not gauge invariant,  $( e^{i \theta} u)^p \neq e^{i \theta } u^p$ for generic $ \theta \in \R$, 
and \eqref{eq:NLS} does not admit an obvious Hamiltonian structure. 
Moreover \eqref{eq:NLS} does not admit any analytic conserved quantities  \cite{jaquette2020global}. 
As such many classical methods of analysis, such as proving global existence via conservation laws, are rendered inapplicable. 

Nevertheless this class of equations has received sustained interest over the past several decades.
Most notably its local well-posedness theory can be extended to negative Sobolev spaces\cite{kenig1996quadratic,bejenaru2006sharp}. 
Recently the quadratic case has been studied as a toy model for the perturbation of a Hamiltonian NLS by an electric field \cite{leger2018global}, and similar nonlinearities appear in other applications  \cite{gustafson2006scattering,colin2009stability,hayashi2013system}.  
For further references we direct the reader to \cite{kishimoto2019remark,jaquette2020global} and the references contained therein.

While \eqref{eq:NLS} may not be Hamiltonian, the complex analyticity of the  nonlinearity imposes a different type of structure.  
Of particular interest to us is that the space of functions supported on non-negative Fourier modes forms an invariant subsystem. 
Moreover in Theorem \ref{prop:Intro_Integrability} we prove this to be an integrable subsystem. 
That is, each Fourier coefficient of a solution can be explicitly solved by quadrature. 
Using these explicit formula we are able to study the dynamics of \eqref{eq:NLS} in great detail. 
Namely,  we show in Theorem \ref{prop:Intro_Quasiperiodic} that small initial data supported on strictly positive Fourier modes will yield (quasi)periodic orbits. 
Moreover these quasiperiodic orbits all have the same frequency vector as solutions to the linear Schr\"odinger equation. 
Hence if $\T^d$ is a rational torus then these quasiperiodic solutions are simply periodic. 
By leveraging this rigidity of allowable frequencies, we prove in Theorem \ref{prop:Intro_Blowup} that sufficiently large monochromatic initial data to the quadratic case of \eqref{eq:NLS} will blow up in finite time in the $L^2$ norm.

To begin our analysis let us consider first the dynamics of spatially constant solutions. 
Such an ansatz yields the complex ODE $i \dot{z} = z^p$, wherein the complex plane is foliated by homoclinic solutions to zero, with the exception of a finite number of solutions which blowup in finite (positive or negative) time. 
Much of this behavior carries over to the whole PDE in \eqref{eq:NLS}. In \cite{jaquette2020global} it is shown that close to constant initial data (of arbitrarily large norm) is guaranteed to have semi-global existence, and moreover solutions will converge to the zero equilibrium  as time limits to either positive or negative infinity. As a consequence there exists an open set of homoclinic orbits from the zero equilibrium to itself.  
If there were to exist any analytic conserved quantity they would necessarily be constant on this open set, and thereby globally constant.

Recent work has  illuminated intricate dynamics in the case of \eqref{eq:NLS} on $\T^1$ with a quadratic nonlinearity. 
From their numerical experiments \cite{Cho2016} observed that real initial data both large and small appears to decay to zero, and they conjectured that real initial data would be globally wellposed.   
Using computer assisted proofs, the work \cite{jaquette2020global}  has demonstrated infinite families of non-trivial equilibria and heteroclinic orbits between these equilibria and the zero equilibrium \cite{jaquette2020global}. 
Many questions remain about the   dynamical behavior of solutions, such as the existence and stability of (quasi)periodic orbits, the possibility of connecting orbits between the non-trivial equilibria, etc. 
At a more foundational level it is unclear when solutions will be globally wellposed. As seen by the ODE $i \dot{z} = z^p$, a smallness condition on the initial data  is by no means sufficient to preclude finite time blowup.

For comparison to another nonlinear Schr\"odinger equation lacking gauge invariance one may consider the equation $ i u_t = \triangle u + |u|^p$. 
By studying the dynamics of the zeroth Fourier mode it has been shown that a robust set of initial data will blowup in finite time
\cite{oh2012blowup,ikeda2015some,fujiwara2017lifespan}.  For the nonlinearity $|u|^2$   recent work has shown that initial data $u_0 \in L^2(\T^1)$ will have a global solution if and only if $ u_0(x) = i \mu_0 $ for $\mu_0 \geq 0$ \cite{fujiwara2020necessary}. 
While global solutions are rare for this PDE, that does not seem to be the case in \eqref{eq:NLS}. 
However without an obvious Hamiltonian structure or conserved quantities it is unclear how to characterize which solutions are globally wellposed.  

It seems counterintuitive that a PDE without analytic conserved quantities would be integrable. 
The term integrable originated to refer to systems whose solutions could be expressed in terms of elementary functions and quadrature, which is to say the definite integrals of known functions, and in Hamiltonian dynamics integrability  coincides with the existence of a complete set of independent integrals of motion \cite{kozlov1983integrability,ramani1989painleve}. 
Integrable PDEs, such as the Korteweg-de Vries equation, are often characterized by possessing an infinite number of conserved quantities, the existence of a Lax pair, and  explicit solvability (e.g. via the inverse scattering transform method).

One such integrable PDE is the cubic Szeg\H{o} equation $ i \partial_t u = \Pi(|u|^2 u)$ which is posed on initial data $u_0 \in L^2_+( \T,\C)$ supported on non-negative Fourier coefficients \cite{gerard2010cubic,gerard2012invariant,gerard2015explicit}, where the Szeg\H{o} projection  $ \Pi : L^2( \T,\C) \to L^2_+( \T,\C)$ is   defined as 
 \[
 \Pi \left(
 \sum_{- \infty < n < \infty} 
 \hat{f}(n) e^{i \omega n x} 
 \right)
 = \sum_{0 \leq n \leq \infty}  \hat{f}(n) e^{i \omega n x}
 , \qquad \qquad 
 f \in L^2( \T,\C) .
 \] 
This equation can be seen as a toy model for totally non-dispersive equations.
In other integrable PDEs, such as the one dimensional defocusing cubic NLS, the infinite hierarchy of conservation laws prevents so-called weak turbulence, which is to say the unbounded growth of a solution's higher Sobolev norms. 
This phenomenon is not prohibited by the cubic Szeg\H{o} equation, which admits solutions whose concentration of  energy undergoes infinitely many transitions alternating from low frequencies to high frequencies, and back to low frequencies \cite{gerard2017cubic,gerard2019survey}.

It is also on the space of non-negative Fourier coefficients that we show in 
Theorem \ref{prop:Intro_Integrability} that \eqref{eq:NLS} possesses an integrable subsystem. 
The invariance of this subsystem can be seen from the fact that the nonlinearity $u^p$ satisfies $ \Pi \circ \Pi(u)^p = \Pi(u)^p$. 
However there cannot exist any analytic conserved quantities on this subspace of non-negative Fourier coefficients by the same argument as in \cite{jaquette2020global}. 
Nevertheless that does not rule out the existence of singular conserved quantities. 
In fact the ODE $ i \dot{z} = z^p$, whose explicit solution is given \eqref{eq:ZeroModeSolution}, preserves the functional $V(z) := \frac{1}{z^{p-1} } + \frac{1}{\bar{z}^{p-1}}$. 
For solutions to \eqref{eq:NLS} supported on non-negative Fourier modes this  functional $V$, when evaluated on a solution's zeroth Fourier mode, is conserved under the dynamics. 
Thus, it does not seem unreasonable to conjecture that there may exist an infinite hierarchy of singular conserved quantities. 
For further  references on singular symplectic forms we direct the  reader to \cite{braddell2019invitation}.

In this paper we will primarily work with a Wiener algebra of functions with summable Fourier coefficients.  
To first fix notation, for $n = (n_1,\dots n_d) \in \Z^d$ define $ | n| = |n_1 | + \dots + | n_d|$. 

\begin{definition}
	For $\omega \in \R^{d} $, the torus $\T^d = \prod_{i=1}^d \R / \frac{2 \pi}{\omega_i} \Z$, a number $s \geq 0$ and $ u \in L^1(\T^d)$ we define a norm 
	\begin{align*}
	\| u \| &= \sum_{n \in \Z^d} (1+|n|)^s | \hat{u}(n) | ,
	&
	\hat{u}(n) &= \frac{1}{\mbox{vol}(\T^d)} \int_{\T^d} u(x) e^{-i \omega n x}.
	\end{align*}
	We define the weighted Wiener algebra $ A_s(\T^d) \subseteq L^1(\T^d)$ to be the set of functions $ u $ with $ \|u\| < \infty$.
\end{definition} 
If one chooses $s=0$, then one obtains the canonical  Wiener algebra of functions with summable Fourier coefficients, which embeds between the spaces  $ Lip(\T^d,\C)  \subseteq A_0(\T^d) \subseteq C^0(\T^d , \C)	$. 
The space  $A_s(\T^d)$ is a Banach algebra, which is to say that $\| u v\| \leq \|u\|\, \|v \|$ for all $ u,v \in  A_s(\T^d)$, and the PDE \eqref{eq:NLS} is locally well posed on $ A_s(\T^d)$  for all $ s \geq 0$.

 To focus on the space of non-negative Fourier coefficients,  we define the following subspace of the Wiener algebra 
 \[
 A^+_{s}(\T^d) = 
 \{
 u \in A_s ( \T^d) :   \hat{u}(n) \neq 0 \mbox{ only if } \min_{1 \leq i \leq d} n_i \geq  0  
 \} .
 \]
 The foundational result we prove in this paper is that we can explicitly compute solutions to the PDE \eqref{eq:NLS} on this subspace, which is to say that   $ A^+_{s}(\T^d)$ is an integrable subspace.

 \begin{theorem} \label{prop:Intro_Integrability}
 	Fix initial data $ u_0 \in A_s^+(\T^d)$ for $ s \geq 0$ and suppose $u(t)$ solves \eqref{eq:NLS}  with initial data $ u(0) = u_0$. 
 	Let $ J = (-T_{min},T_{max})$ denote the maximal interval of existence and define functions $a_n : \R\to \C$ such that  
 	\[
 	u(t,x) = \sum_{n \in \N^d} a_n(t) e^{i \omega n x}
 	\]
 	for all $ t \in J$. 
 	Then each  function $ a_n(t)$ may be solved for explicitly by quadrature.
 \end{theorem}
 
 If the initial data $u_0 \in A_s^+(\T^d)$ has zero average, then the formula for computing solutions simplifies, as we present in Corollary \ref{thm:ExplicitSolution}. As an example, we present this result for the quadratic NLS posed on $\T^1$. 
\begin{example}	
	\label{eg:Recursive}
	Consider \eqref{eq:NLS} with $(p,d)=(2,1)$ and fix initial data $ u_0 \in A_0^+(\T^1)$.   
	Let $ J = (-T_{min},T_{max})$ denote the maximal interval of existence and define functions $a_n : \R\to \C$ such that  
	\[
	u(t,x) = \sum_{n=0}^\infty  a_n(t) e^{i \omega n x}
	\]
	for all $ t \in J$. Let  $ \phi_n = \hat{u}_0(n)$ for all $ n \in \N$. If $ \phi_0 = 0$ then the functions $a_n $ are recursively defined as 
	\begin{align*} 
	a_0(t) &:=  0 , \\
	a_n(t) &:= \phi_n e^{i \omega^2 n^2 t}  -  i e^{i \omega^2 n^2 t} 
	\int_0^t  e^{-i \omega^2 n^2 s} \sum_{1\leq k \leq n-1} a_{k}(s) a_{n-k}(s) ds. 
	\end{align*}
\end{example}

This style of recursively defining the Fourier coefficients is similar to previous results in the literature studying other complex PDEs with quadratic nonlinearity. 
The work \cite{sakajo2003blow} used such an approach to study blowup solutions of the Constantin-Lax-Majda  equation with periodic data. 
A similar approach appears in  \cite{guo2013convergence} to study $ u_t = \triangle u + u^2$ and in 
\cite{bejenaru2006sharp,iwabuchi2015ill} to study $ i u_t + \triangle u = u^2$, each considering complex  initial data $u_0 :\R^d \to \C$.  

As an illustrative example of the functions $a_n(t)$ defined in Example \ref{eg:Recursive}  we consider the case of monochromatic initial data.  
\begin{example} \label{eg:ExplicitFunctions}
	Fix $ A \in \C$ and initial data $ u_0(x) = A e^{i \omega x}$. 
	The first few functions $ a_n : \R \to \C$ defined as in Example \ref{eg:Recursive} are given as follows 
	\begin{align*}
		a_1(t) &= A e^{i \omega^2 t} \\
		%%%%%%%%%%%%%%%
		a_2(t) &= \frac{A^2}{\omega^2} 
		\left( \frac{1}{2} e^{2 i \omega^2  t}  - \frac{1}{2} e^{4 i \omega^2  t}  \right)\\
		%%%%%%%%%%%%%%%
		a_3(t) &=  \frac{A^3}{\omega^4}
		\left( \frac{1}{6} e^{3 i \omega^2  t}
		- \frac{1}{4} e^{5 i \omega^2  t} 
		+ \frac{1}{12} e^{9 i \omega^2  t}  \right) \\
		%%%%%%%%%%%%%%%
		a_4(t) &=  \frac{A^4}{\omega^6}
		\left(
		\frac{7 e^{4 i \omega ^2t}}{144 }-\frac{e^{6 i  \omega ^2t}}{10 }+\frac{e^{8 i  \omega ^2t}}{32 }+\frac{e^{10 i  \omega ^2t}}{36}-\frac{11 e^{16 i  \omega ^2t}}{1440 } 
		\right)  \\
		%%%%%%%%%%%%%%%%
		a_5(t) &= 	
		\frac{A^5}{\omega^8}	
				\left(
		\frac{19 e^{5 i t \omega ^2}}{1440 }-\frac{37 e^{7 i t \omega ^2}}{1080  }+\frac{5 e^{9 i t \omega ^2}}{256}+\frac{5 e^{11 i t \omega ^2}}{504  }-\frac{e^{13 i t \omega ^2}}{144 }-\frac{11 e^{17 i t
				\omega ^2}}{5760 }+\frac{113 e^{25 i t \omega ^2}}{241920  }
			\right)  
%			\\			& \;\; \vdots
	\end{align*}
\end{example}

Note that each function $ a_n(t)$ is given by   $ A^n/ \omega^{2(n-1)}$, multiplied by a function which is $ 2 \pi /\omega^2$ periodic in time.  
This is shown to be the case in general in Theorem \ref{prop:Rescaling}. 
If the ratio $ |A|/\omega^2$ is small, we expect $ \| a_n\|_{L^\infty}$ to shrink geometrically in $n$, and expect the solution $u$ to be $2 \pi / \omega^2$ periodic. Likewise, if the ratio is large, then we would expect the coefficients to grow geometrically in $n$, and the solution  $u$ to blowup. 
 In Figure \ref{fig:Solutions} we plot  solutions for $\omega =1$ and values $ A =1,2,3$   with the  functions $a_n(t)$ computed up to $ n =300$. 
 Numerical evidence suggests that the ratio $|A|/\omega^2 \approx 3.37$ is the approximate dividing line between monochromatic initial data whose solution is periodic versus blows up in finite time. The following theorem summarizes what we can prove rigorously. 
 
 \begin{figure}[t!]
 	\centering
 	\includegraphics[width=1\textwidth]{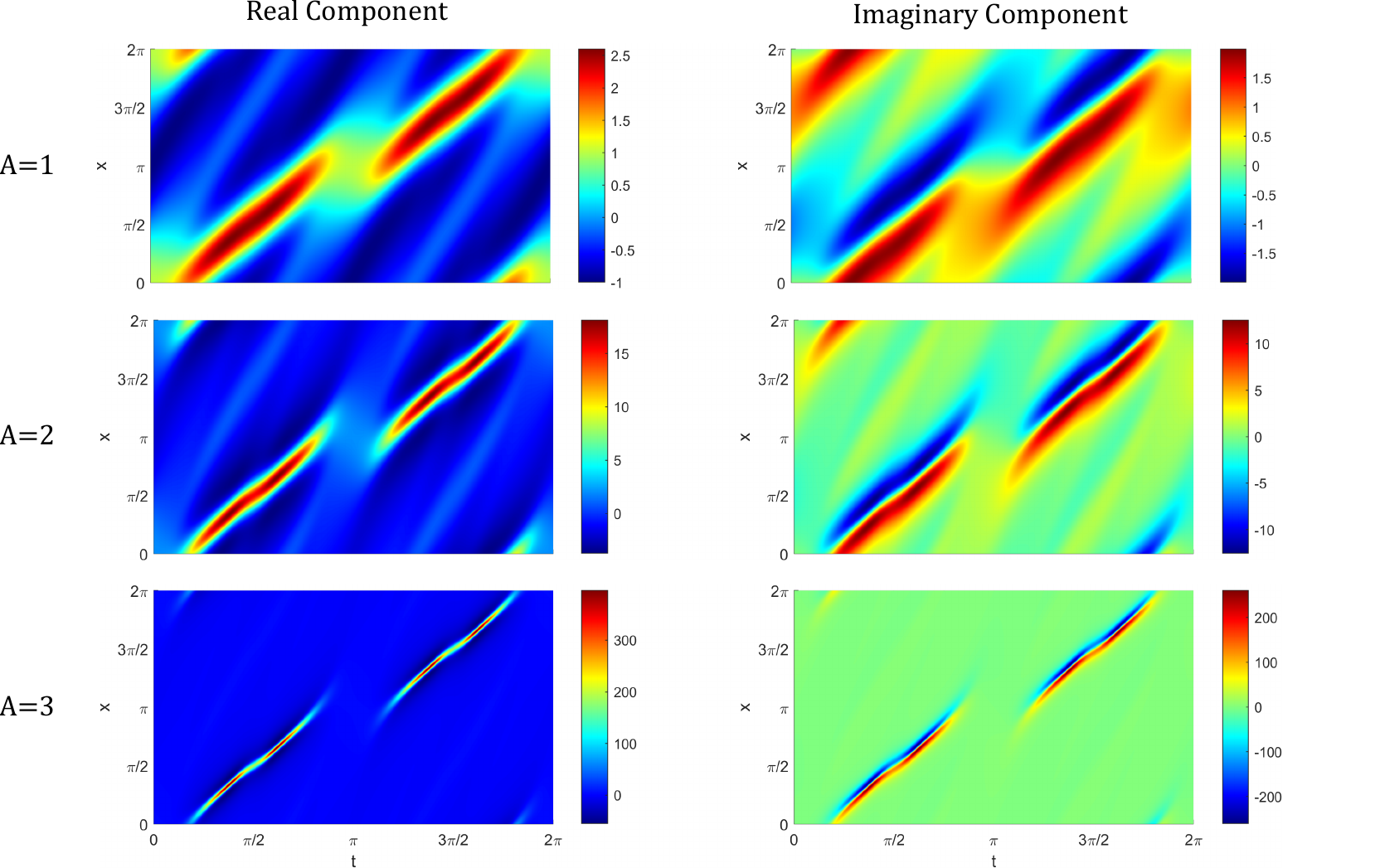} 
 		\caption{Plots of periodic solutions $u(t,x)$ with initial data $u_0(x) = A e^{i x}$.  
 	}
 	\label{fig:Solutions}
 \end{figure}

\begin{theorem} \label{prop:Intro_Blowup}
	Consider   \eqref{eq:NLS}  with $(p,d)=(2,1)$ and  	   initial data $u_0(x) =   A e^{ i \omega  x}$ with  $A \in \C$,  $ \omega >0$.   
	\begin{enumerate}[(a)]
		\item If $|A| \leq  3  \omega^2$, then  $u \in C^0( \R , A_0^+(\T^1))$ is a periodic solution with period $\frac{2 \pi}{\omega^2}$. 
		\item If $|A| \geq  6\omega^2$,  then the solution blows up in finite time in the $L^2$ norm. In particular, there exists some  $|T_*| \leq  \frac{2 \pi}{\omega^2}$ such that  $\limsup_{t \to T_*} \| u(t)\|_{L^2} = + \infty$.   
	\end{enumerate}
\end{theorem}

The existence of periodic solutions is not limited to the special choice of monochromatic initial data; 
we show that for moderately sized initial data supported on positive Fourier modes, their solutions will be (quasi)periodic. 
\begin{theorem}
	\label{prop:Intro_Quasiperiodic}
	Consider \eqref{eq:NLS}   
	with initial data $u_0 \in A_s^+(\T^d)$ supported on strictly positive Fourier modes,  that is $ \hat{u}_0(n) \neq 0 $ only if $ n_{i} \geq 1 $ for all  $ 1 \leq i \leq d$.  If  
	\[
	\|u_0 \|  \leq    \frac{p-1}{p} \left( \frac{\|\omega\|^2 (p-1)}{2}\right)^{1/(p-1)},
	\]
	where $\|\omega\|^2  = \sum_{1 \leq i \leq d} \omega_i^2$, 
	then the solution $u \in C^0(\R, A_s^+(\T^d))$ is (quasi)periodic with  frequencies $\{ \omega_i^2 / 2 \pi  \}_{i=1}^d$ and $\| u(t) \| \leq  \left( \frac{\|\omega\|^2 (p-1)}{2}\right)^{1/(p-1)}$ for all $ t\in \R$.   
\end{theorem}

It is interesting to note that the (quasi)periodic orbits in Theorem \ref{prop:Intro_Quasiperiodic} all have locked frequencies. 
If $\T^d$ is a rational torus then these solutions are simply periodic, and if the set  $\{\omega_i^2\}_{i=1}^d$ is rationally independent then each solution densely wraps around an invariant torus of dimension $d$. 
This frequency locking makes the problem of characterizing (quasi)periodic solutions extremely resonant; each of the periodic orbits in Theorem \ref{prop:Intro_Blowup} has a Floquet multiplier of $1$ with infinite multiplicity.

This resonance may be seen in contrast to the Hamiltonian NLS $ i u_t = \triangle u \pm |u|^{p-1} u$ where monochromatic initial data yields periodic solutions with amplitude dependent frequency. 
This may also be seen in contrast to the nonresonance conditions one typically imposes in the study of quasiperiodic orbits to overcome the problem of small divisors.
In Hamiltonian PDEs there are two widespread methods for proving existence of quasiperiodic orbits: one based on a Lyapunov-Schmidt reduction \cite{craig1993newton,bourgain1998quasi} and the other based on the KAM method \cite{kuksin1996invariant,procesi2015kam}. 
While the extreme resonance in the frequencies of the quasiperiodic solutions to \eqref{eq:NLS} may seem daunting, we are able to leverage the explicit formula for solutions in Theorem \ref{prop:Intro_Integrability} to great effect, and  circumvent the delicate analysis required of the standard  methods.

As can be seen from Example \ref{eg:Recursive}, we have a formula for computing a solution's Fourier coefficients to arbitrary order. However there is no \emph{a priori} guarantee that this series will converge for all $ t \in \R$.  
Indeed, the monochromatic initial data $u_0(x) = A e^{i \omega x}$ in Theorem \ref{prop:Intro_Blowup} gives a parameterized family of periodic orbits which terminates in a finite time blowup if $ |A| / \omega^2 $ is too large. 
As a result of each periodic orbit having period $2 \pi / \omega^2$, the blowup time is bounded above by $ 2 \pi / \omega^2$. 
The periodic orbits in Figure \ref{fig:Solutions} suggests that the blowup will occurs at single point $ x \in \T^1$. 
Unfortunately  we cannot hope for the existence of a simple blowup profile, as the only solutions of the form $ u(t,x) = \psi(t) \varphi(x)$ are either constant in $t$ or $x$ \cite{jaquette2020global}. 

We have primarily studied solutions with initial data supported on strictly positive Fourier modes  
as this avoids the introduction of secular terms. 
This is not the case for generic initial data in $ A_0^+(\T^d)$, such as those with a nontrivial zeroth Fourier mode. 
It should be noted that, as Theorem \ref{prop:Intro_Integrability} allows us to exactly compute solutions, this appearance of secular terms are not an artifact of a perturbative ansatz, but an aspect  of the genuine solution.  
While the local wellposedness of these solutions is guaranteed, a complete understanding of the dynamics in $A_0^+(\T^d)$  is not clear.

More broadly, it is interesting to ask how perturbations of solutions in $A^+_0(\T^d)$ behave in the whole space $A_0(\T^d)$. 
For example, what is the local stability of the  (quasi)periodic orbits in Theorem \ref{prop:Intro_Quasiperiodic}? 
Furthermore, if some initial data $u_0 \in A_0^+(\T^d)$ blows up in finite time, will perturbations of it in $A_0(\T^d)$  also blowup?

The rest of the paper proceeds as follows. 
In Section \ref{sec:Main} we begin by proving how solutions may be solved by quadrature in  Theorem \ref{prop:Intro_Integrability}, and then work towards proving the existence of (quasi)periodic orbits in  Theorem \ref{prop:Intro_Quasiperiodic}. 
To do so, first we derive a recursive relationship for the space-time Fourier coefficients of a solution in Lemma \ref{prop:FF_Recursive}. 
We then equate solutions to this recursive relationship to fixed points of an operator $T$ given in Definition \ref{def:Newton_Map}. 
To prove Theorem \ref{prop:Intro_Quasiperiodic}, we then show that $T$ has a fixed point in a suitably chosen ball using the Schauder fixed point theorem.

Then in Section \ref{sec:Monochromatic} we focus on proving Theorem \ref{prop:Intro_Blowup}. 
First we show in Theorem \ref{prop:Rescaling} that solutions of \eqref{eq:NLS} with monochromatic initial data satisfy a geometric scaling in their Fourier coefficients. 
We then prove    part (b) of Theorem \ref{prop:Intro_Blowup} that sufficiently large monochromatic initial data will blowup in the $L^2$ norm. 
To  prove part (a)  we use a computer assisted proof to explicitly compute the coefficients up to a certain order, and then show that the operator $T$ has a fixed point in a ball about this truncated solution.

\section{Integrability and Quasiperiodicity}
\label{sec:Main}

In this paper, we study initial data to equation \eqref{eq:NLS} with non-negative Fourier modes. 
We begin by reviewing some algebraic properties of $\N^d$, where  $ \N$ denote the set of non-negative integers. 
Firstly the set $ \N^d$ forms a commutative semi-ring, which is to say it inherits the algebraic operations of elementwise addition and multiplication from $ \Z^d$,  however $\N^d$ lacks additive inverses. 
The space $\N^d$ may also be endowed with the structure of a semi-module, which is to say that if $ r \in \N$ and $n = ( n_1, \dots ,n_d) \in \N^d$ then we define $ r n = ( r n_1, \dots , r n_d)$. 
Furthermore, to clarify a slight abuse of notation, for $ r \in \N$ and $n \in \N^d$ we define $ r+n = r 1_{\N^d} + n = (r+n_1 , \dots ,r+ n_d)$ where $1_{\N^d} = (1 , \dots ,1)$ denotes the multiplicative identity on $\N^d$. 

We define the Banach space 
\begin{align*}
\ell^1_{s,d} &:= \left\{ \{ a_n \}_{n\in\N^d}: a_n \in \C,  \|a\|_{\ell^1_s} < \infty   \right\} ,&
\|a\| := \sum_{n\in \N^d} 	(1+|n|)^s
|a_n |,
\end{align*}
 and   the Cauchy product  of two elements $ a , b \in \ell^1_{s,d}$ by 
\begin{align*}
(a * b)_n &:= \sum_{\substack{n_1 + n_2 = n  \\ n_1,n_2 \in \N^d}} a_{n_1} b_{n_2} .
\end{align*}
Note that for each fixed $ n \in \N^d$ the sum above is finite, and the Cauchy product here is naturally identified the multiplication of two power series in $d$ variables.
The $d$ variables we consider here are $ e^{i \omega_k  x_k}$ for $ 1 \leq k \leq d$.

The sequence space $ \ell^1_{s,d}$ is isometrically isomorphic to $ A_s^+(\T^d)$. 
In this manner the Cauchy product of two elements in  $ \ell^1_{s,d}$ corresponds with the multiplication of their respective functions $ A_s^+(\T^d)$.  
Furthermore $\ell^1_{s,d}$ is a commutative Banach algebra;  if $a,b\in \ell^1$, then $ \| a * b\| \leq \| a\| \| b\|$. 
For notational brevity  we define $ a^p = a * \dots *a$. 	
Also note that the function $F(a) = a^p$ is Fr\'echet differentiable and $ DF(a)h = p a^{p-1} * h$.

Below we show that if $ u_0 \in A^+_0(\T^d)$ then  each of the Fourier modes $a_n(t)$ of its solution can be explicitly solved  by quadrature.

\begin{proof}[Proof of Theorem  \ref{prop:Intro_Integrability}]
	Fix initial data $ u_0 \in A^+_s(\T^d)$ and define $ \phi_n \in \ell_{s,d}^1$ by $ \phi_n = \hat{u}_0(n)$. 
	Let  $u(t)$  denote the solution to \eqref{eq:NLS} with $ u(0) = u_0$, and $J\subseteq \R$   the maximal interval of existence. 
	As $ A_0^+(\T^d)$ and $ \ell_{0,d}^1$ are isometrically isomorphic, there exists a corresponding function  $ a: J \to \ell^1_{s,d}$ such that  
	\[
	u(t) = \sum_{n \in \N^d} a_n(t) e^{i \omega n x}
	\]
	By expanding the PDE \eqref{eq:NLS} in Fourier modes and matching terms, we see that $a$ must satisfy the following densely defined differential equation, defined in each component as: 
	\begin{align}
	\label{eq:FourierEquation}
	\dot{a}_n =  i \omega^2 n^2 a_n - i (a^p)_n .
	\end{align}
	Note that the zero mode of  \eqref{eq:FourierEquation} becomes as written below  
		\begin{align*}
		\dot{a}_0  &= 	-i a_0^p
	\end{align*}
	which has the following explicit solution  
\begin{align}\label{eq:ZeroModeSolution}
	a_0(t) := \frac{\phi_0}{(1+ i(p-1) \phi_0^{p-1} t)^{1/(p-1)}}  
\end{align}
where one takes the branch $ 1^{1/(p-1)} = 1$ so that $ a_0(0) = \phi_0$.

We will prove by induction that each function $a_n(t)$ can be solved by quadrature. 
Fix $n \in \N^d$ with $|n| \geq 1$ and assume that for each $ k \in \N^d$ with  $ 0 \leq |k|  \leq |n|-1$ that $a_k(t)$ may be solved for by quadrature. The differential equation \eqref{eq:FourierEquation}   is linear in $a_n$, so to solve the equation, we define integrating factors 
\begin{align*}
	P_n(t) &=  i \omega^2 n^2 -  i p (a_0(t))^{p-1}  , & 
	Q_n(t) &= -i\big(
	(a(t)^p)_n -  p a_0(t)^{p-1} a_n(t)   \big)\\
	&&
	&=- i 
	\sum_{
		\substack{n_1 + \dots + n_p = n  \\
			0 \leq 	|n_1|,\dots,|n_p| < |n|
	}}
	a_{n_1}(t)  a_{n_2}(t) \cdots 	 a_{n_p}(t) .
\end{align*}
As defined, $Q_n(t)$ depends only on the Fourier modes $a_k(t)$ for $ |k| < |n|$ where each $a_k(t)$ may be solved for explicitly by quadrature.  
Substituting $P_n$ and $Q_n$ back into  \eqref{eq:FourierEquation}  yields  the following non-homogeneous linear first order differential equation  
\[
\dot{a}_n =  P_n(t) a_n + Q_n(t).
\]
	Then $a_n(t)$ can be explicitly solved for by quadrature as follows: 
	\begin{align}
		\label{eq:ExplicitFormula}
		a_n(t) = 
		\phi_n \exp \left\{  \int_0^t P_n(s) ds \right\} 
		+
		\exp \left\{  \int_0^t P_n(s) ds \right\} 
		\int_0^t \exp \left\{ - \int_0^s P_n(\tau) d \tau \right\} 
		Q_n(s)  ds .
	\end{align}

\end{proof}

To reiterate, while the expression $(a(t)^p)_n$ in the definition of $Q_n$ may appear to depend on the entire sequence $\{ a_n(s)\}_{n \in \N^d} $, it in fact only depends on components $a_k(s)$ with $ | k | \leq |n|$.
In the definition of $Q_n$ we subtract away the dependence on $a_n(t)$, thereby $Q_n$ only depends on the solutions  $a_k(t)$ for $ |k| < |n|$. 
For instance, in Example \ref{eg:Recursive} we  have  $ Q_n(t) = -i  \sum_{k=1}^{n-1} a_{k}(s) a_{n-k}(s)$.
In this manner the functions $a_n(t)$ may be recursively solved for using the formula in \eqref{eq:ExplicitFormula} for increasing orders of $|n|$. 

As a further remark, note that 	\eqref{eq:ExplicitFormula} is essentially  the Duhamel's formula version of the differential equation in \eqref{eq:FourierEquation}. 
If  $\phi_0 = 0$ we obtain a significant simplification.
\begin{corollary}\label{thm:ExplicitSolution}
		Fix initial data $ u_0 \in A_s^+(\T^d)$ and suppose $u(t)$ solves \eqref{eq:NLS}  with initial data $ u(0,x) = u_0(x)= \sum_{n \in \N^d}^{\infty}  \phi_n e^{ i \omega n x}$ and suppose  $ \phi_0=0$. 
	Let $ J = (-T_{min},T_{max})$ denote the maximal interval of existence and define  a function $ a: J \to \ell_{s,d}^1$ such that
	\[
	u(t,x) = \sum_{n\in \N^d} a_n(t) e^{i \omega n x}
	\]
	for all $ t \in J$. 
	The $a_n(t)$ are given by the recursive formula : 
	\begin{align} \label{eq:Bdef}
	a_0(t) &= 0,\\
%	a_1(t) &:= \phi_1 e^{i \omega^2 t} ,\\
	a_n(t) &= \phi_n e^{i \omega^2 n^2 t}  -  i e^{i \omega^2 n^2 t} 
	\int_0^t  e^{-i \omega^2 n^2 s}  \left(
  a^p(s)  
	\right)_n  ds, 
	\label{eq:BdefIndunction}
	\end{align}
where  we define  $ \omega^2 j = \sum_{i=1}^d \omega_i^2 j_i  \in \R$ for $ \omega \in \R^d$ and $ j \in \N^d$.  

\end{corollary}  
\begin{proof}
	If $\phi_0 = 0$, then $a_0(t)=0$, $P_n(t) = i \omega^2 n^2$ and $ Q_n(t) = -i (a^p)_n$.  The corollary follows by simplifying the expression in \eqref{eq:ExplicitFormula}.  
\end{proof}

The computation from Example \ref{eg:ExplicitFunctions} in the case $(p,d)=(2,1)$ suggests that small monochromatic initial data will yield periodic solutions with a fixed frequency of $ \omega^2/2\pi$.  
To study such functions we define the space $\ell^1_{s,d} \otimes \ell^1_{0,d}$ of space-time Fourier coefficients; for a sequence $ c \in \ell^1_{s,d} \otimes \ell^1_{0,d}$ we   associate it with a (quasi)periodic function $u \in C^0( \R , A_s^+(\T^d))$ given as  
\begin{align}\label{eq:FF_function_correspondance}
	u(t,x) &= \sum_{n,j \in \N^d}
	c_{n,j} e^{i \omega^2 j t} e^{i n \omega  x} .
\end{align}
We define a norm on $ \ell^1_{s,d} \otimes \ell^1_{0,d}$ as  
\[
\|a \| 
 = \sum_{n,j \in \N^d }  (1+|n|)^s | a_{n,j}| .
\]
Note that if $ s=0$, then $ \ell^1_{s,d} \otimes \ell^1_{0,d}$ is isometrically isomorphic to $\ell^1_{0,2d}$. 
Again we may define the Cauchy  product, corresponding with the multiplication of two functions in  $C^0( \R , A_s^+(\T^d))$, 
for  two elements $a ,  b \in \ell^1_{s,d} \otimes \ell^1_{0,d}$  as follows  
	\[
	(a * b)_{n,j} := \sum_{\substack{n_1 + n_2 = n  \\ n_1,n_2 \in \N^d}}
	\sum_{\substack{j_1 + j_2 = j  \\ j_1,j_2 \in \N^d}}
	a_{n_1,j_1} b_{n_2,j_2} .
	\]
	This space is also a Banach algebra;  for $ a,b \in \ell^1_{s,d} \otimes \ell^1_{0,d}$ we have  
	$ \| a * b \| \leq \| a  \| \, \| b \| $. 	 
	
	One may further note from  Example \ref{eg:ExplicitFunctions} that $a_n(t)$ contains Fourier modes $ e^{i \omega^2 j t}$ only for $ n \leq j \leq n^2$. 
	This example was only for the one dimensional case $ x \in \T^1$.  
	 To make sense of an expression $n \leq j \leq n^2$ for elements $ n,j \in \N^d$, we  introduce a partial order  on $ \N^d$. 
	\begin{definition}
Define a partial order $\leq$ and a strict partial order $<$ on $\N^d$ as follows. 	
		\begin{itemize}
			\item  Two elements $ m,n \in \N^d$ satisfy $m \leq n$ if and only if $ m_k \leq n_k$ for all $1 \leq k \leq d$. 
			\item Two elements $ m,n \in \N^d$ satisfy $ m < n $ if and only if $m_k < n_k$ for all $1 \leq k \leq d$.  
		\end{itemize}
	\end{definition} 

 These partial orders are respected under addition and multiplication of elements in $\N^d$. That is for elements $a,b,c,d \in \N^d$, if $ a \leq b$ and $c \leq d $, then $ a + c \leq b + d$ and $ a \cdot c \leq b \cdot d$. 
 Furthermore we note that if $ m \leq n $ for elements $ m,n \in \N^d$ then $ n - m \in \N^d$.   
 We are now prepared to define a subspace of  $\ell^1_{s,d} \otimes \ell^1_{0,d}$ which generalizes the type of Fourier series seen in Example \ref{eg:ExplicitFunctions}. 
\begin{definition}
	Define the Banach space 
	\[
	X := \left\{ 
	c \in \ell^1_{s,d} \otimes \ell^1_{0,d} :   c_{n,j} \neq  0 \mbox{ only if } 1_{\N^d} \leq n \mbox{ and } n \leq j \leq n^2
	\right\}
	\]

\end{definition}
The space $X$ inherits its norm and Banach algebra structure from $ \ell^1_{s,d} \otimes \ell^1_{0,d}$, and the following lemma shows that  $*:X \times X \to X$ is a closed operation. 
 
\begin{lemma}\label{prop:Xclosed}
	Suppose  $ c \in X$ and $ n,j \in \N^d$. Then $ (c^p)_{n,j} \neq 0  $ only if $p \leq n$ and $ n \leq j \leq n^2 - (p-1)(2n-p)$. 
\end{lemma}
\begin{proof}
	Fix $ n , j \in \N^d$. 
	The product $ c^p$ may be expressed as  
	\[
	(c^p)_{n,j} = 
	\sum_{n_1 + \dots + n_p=n} 
	\sum_{\substack{ 
		j_1 + \dots + j_p =j \\
		n_{i} \leq j_i \leq n_{i}^2 \\
		1 \leq i \leq p
	}   }
	c_{n_1,j_1} \dots c_{n_p , j_p}
	,
	\qquad \qquad 
	n_1, \dots n_p,j_1 , \dots j_p \in \N^d .
	\]
	
	As $c \in X$, then  we have $ \prod_{1\leq i\leq d} c_{n_i,j_i} \neq 0$ only if $ 1_{\N^d} \leq n_i$ for all $1 \leq i \leq p$. 
	The partial order on $\N^d$ respects addition, so if $ n_i \geq 1_{\N^d} $ for each $1 \leq i \leq p$   then  $n = n_1 + \dots + n_p \geq p 1_{\N^d}$.  
	Hence $(c^p)_{n,j} \neq 0$ only if $ n \geq p 1_{\N^d} $. 
	
	We next prove the lower bound that $ (c^p)_{n,j} \neq 0 $ only if $ n \leq j$.  
	Let us fix partitions $ n_1 + \dots + n_p = n$ and $j_1 + \dots + j_p =j$ and suppose that $ 	c_{n_1,j_1} \dots c_{n_p , j_p} \neq 0$. 
	As $c \in X$, then $ n_i \leq j_i \leq n_i^2$ for all $ 1 \leq i \leq p$.  
	Again, since the partial order $\leq $ on $\N^d$ respects addition, it follows that 
	\[
	n = n_1 + \dots + n_p \leq j_1 + \dots + j_p = j .
	\]
	 Hence $	c_{n_1,j_1} \dots c_{n_p , j_p} \neq 0 $, and moreover $ (c^p)_{n,j} \neq 0$,  only if $n \leq j$.  
	
	To prove the upper bound, again 	let us fix partitions $ n_1 + \dots + n_p = n$ and $j_1 + \dots + j_p =j$ and suppose that $ 	c_{n_1,j_1} \dots c_{n_p , j_p} \neq 0$. 
	As $c_{n_i,j_i} \in X$ for all $ 1 \leq i \leq p$ then 
\begin{align}
		j &= j_1 + \dots + j_p \leq n_1^2 + \dots + n_p^2 . \label{eq:XclosedUpperBound}
\end{align} 
	To proceed, we will focus on one component $(n)_k$ of $n \in \N^d$ for  fixed $1 \leq k \leq d$; define $ m = (n)_k$ and $ m_i = (n_i)_k$ for each $1 \leq i \leq p$. 
	Then the  inequality in \eqref{eq:XclosedUpperBound} becomes 
\begin{align} \label{eq:XclosedUpperBoundReduced}
		(j)_k  &\leq m_1^2 + \dots + m_p^2 ,
\end{align}
	where $ m_1 + \dots + m_p = m \in \N$ and $ 1 \leq m_i$ for all $ 1 \leq i \leq p$.     
	
	 To define an upper bound on \eqref{eq:XclosedUpperBoundReduced} for all possible summations, we define   $J^+_p: \N\to \N$  below
	\[
	J^+_p(m) := \max 
	\left\{
	m_1^2  + \dots + m_p^2  \in \N
	\, \Big| 
	m_1 + \dots + m_p = m  \mbox{ and } m_i \geq 1   \mbox{ for } 1 \leq i \leq p
	\right\} .
	\]
	We show by induction on $p$ that
	\begin{align}
J^+_p(m) &= \underbrace{1^2+ \cdots  +1^2}_{p-1 \text{ times} }
+ (m-(p-1))^2
 \nonumber \\
&=m^2 - (p-1)(2m-p) .\label{eq:JplusDef}
	\end{align} 
	If $ p =1$ then $J^+(m) = m^2$, hence the base case is satisfied.  
	We make the inductive assumption that $J^+_q(m)$ satisfies \eqref{eq:JplusDef} for all $ q < p$ and all $ m \in \N$. 
	We may then compute  
\begin{align}
	J_p^+(m) &= \max_{ \substack{m_1 + \dots + m_p =m\\ 1 \leq m_1,\dots , m_p \leq m }} m_1^2 + \dots + m_p^2 
	\nonumber
	\\
		&= 
	\max_{1\leq m_1 \leq m-(p-1) }
	\left( m_1^2 + 
	\max_{\substack{ m_2 + \dots +m_p = m-m_1  \\ 1 \leq m_2 , \dots , m_p \leq m-m_1}}
	m_2^2 + \dots + m_p^2 \right)
	\nonumber
	\\
	&= 
	\max_{1\leq m_1 \leq m-(p-1) }
	\left( m_1^2 + 
	J_{p-1}^+(m-m_1) \right)
	\nonumber
	\\
	&= 
	\max_{1\leq m_1 \leq m-(p-1) } m_1^2 + \left( (m-m_1)^2 -(p-2)( 2(m-m_1) -(p-1)) \right) 
	\label{eq:Xalgebra1}
	\\
	&= \left(
	m^2-(p-1)(2m-p)
	\right) 
	+
	\max_{1\leq m_1 \leq m-(p-1) } 
	-2 (m_1-1)\left( 
	m-(p-1)-m_1
	 \right) .
	 	\label{eq:Xalgebra2}
\end{align}
Here \eqref{eq:Xalgebra1} was obtained by the inductive step,  and \eqref{eq:Xalgebra2} was obtained by adding and subtracting the expression in \eqref{eq:Xalgebra1} when $m_1 = 1$. 
The quantity inside the maximum in \eqref{eq:Xalgebra2} is non-positive for $ 1 \leq m_1 \leq m-(p-1)$, and is maximized at the value $0$ when either $m_1 = 1 $ or $m_1 = m-(p-1)$. Thus, the inductive step is completed.

As inequality \eqref{eq:XclosedUpperBoundReduced} must be satisfied, then $(j)_k \leq J_p^+(m) = m^2 - (p-1)(2m-p)$ where $m = (n)_k$. 
Hence $	c_{n_1,j_1} \dots c_{n_p , j_p} \neq 0 $, and moreover $ (c^p)_{n,j} \neq 0$,  only if $j \leq n^2 - (p-1)(2n-p)$. 
The lemma follows.

\end{proof}

 Note that  Lemma \ref{prop:Xclosed} only dealt with finite sums and did not require any convergence properties. We may define the following space. 
  \begin{definition}
 	Define the vector algebra  
 	\[
 	\cX := \left\{ 
 	c \in \C^{\N^d} \otimes \C^{\N^d} : c_{n,j} \neq  0 \mbox{ only if } 1_{\N^d} \leq n \mbox{ and }  n \leq j \leq n^2
 	\right\}
 	\] 
 \end{definition}
Then Lemma \ref{prop:Xclosed} follows when we replace $ X $ by $ \cX$. Hence it follows that $ *: \cX \times \cX\to \cX$  is a well defined operation.   
We now present a refinement of Corollary \ref{thm:ExplicitSolution} for initial data on strictly positive Fourier modes, showing that the functions  $a_n(t)$ are given as a Fourier sum.

\begin{lemma}  \label{prop:FF_Recursive}
	Take the same hypothesis as in Corollary \ref{thm:ExplicitSolution}, having fixed some initial data $ \phi \in \ell^1_{s,d}$. In addition, assume that   $\phi_{n} \neq  0 $ only if $ 1_{\N^d} \leq n$.  Then there exist unique Fourier coefficients $c \in \cX$ for which    the functions $a_n$  are given by:
	\begin{align} \label{eq:PeriodicCoefficients}
	a_n(t) 
	&=
	\sum_{n \leq j \leq n^2} 
	c_{n,j} e^{i \omega^2 j t},
	\end{align}
	where the Fourier coefficients $ c_{n,j}$ may be recursively defined as $c_{0,j} = 0$ for all $ j \in \N^d$, and 
\begin{align} 
	c_{n,j} &= 
	\begin{dcases} 
		\frac{(c^p)_{n,j} }{		\omega^{2}(n^2-j) \;\;}
		& \mbox{ if } n\leq j<n^2  \\
		\phi_n -
		\sum_{n \leq k < n^2} \frac{(c^p)_{n,k} }{		\omega^{2}(n^2-k) \;\;}
		& \mbox{ if } j = n^2 ,
	\end{dcases}
\label{eq:FF_Recursive}
\end{align}
where  we define  $ \omega^2 j = \sum_{i=1}^d \omega_i^2 j_i  \in \R$  for $ \omega \in \R^d$ and $ j \in \N^d$. 

\end{lemma}

We remark that for all $n,j\in \N^d$ that the product  $(c^p)_{n,j}$ only depends on coefficients $c_{m,k}$ for $|m| < |n|$.  Moreover $(c^p)_{n,j}$   only depends on coefficients $ c_{m,k}$ for which $ m \leq n - (p-1) < n$. 
Thus  we can solve $c_{n,j}$ in  increasing orders of $|n|$.

\begin{proof}
	From Corollary \ref{thm:ExplicitSolution} we have 
	$a_0(t) = 0$ and furthermore $ a_n(t) =0$ for all $ |n| < d$. 
	Hence $ c_{n,j} =0$ for all $|n|<d$. 
	
	We prove the rest by induction, assuming $a_m(t)$ has the form in \eqref{eq:PeriodicCoefficients} for all $ m < n$. 
	By Lemma \ref{prop:Xclosed}  it   follows that  
	\[
(	a^p(s))_n =
	\sum_{ 
		n \leq j \leq n^2 - (p-1)(2n-p)
	}
	(c^p)_{n,j} e^{i \omega^2 j s} .
	\]
Again 	by Lemma \ref{prop:Xclosed} we have $ (a(s)^p)_n  \neq 0$ only if $ n \geq p$. 
By plugging the above expression  into equation  \eqref{eq:BdefIndunction} we obtain  
	\begin{align*}
	a_n(t) &= \phi_n e^{i \omega^2 n^2 t}  - i e^{i \omega^2 n^2 t} 
	\int_0^t 
	e^{-i \omega^2 n^2 s}
	\left(
	\sum_{ 
		 n \leq j \leq n^2 - (p-1)(2n-p)
	}
	(c^p)_{n,j} e^{i \omega^2 j s}
	\right)
	ds 
	\nonumber
	\\
	%%%%%%%%%%%%%%%%%%%%%%%%%%%%%%%%
	&= 
	\phi_n e^{i \omega^2 n^2 t}  - i e^{i \omega^2 n^2 t} 
		\sum_{ 
		n \leq j \leq n^2 - (p-1)(2n-p)
	}
	\int_0^t 
	(c^p)_{n,j} e^{-i \omega^2 (n^2-j) s}
	ds 
\\
	%%%%%%%%%%%%%%%%%%%%%%%%%%%%%%%%
		&= 
	\phi_n e^{i \omega^2 n^2 t}  - i e^{i \omega^2 n^2 t} 
	\sum_{ 
		n \leq j \leq n^2 - (p-1)(2n-p)
	}
(c^p)_{n,j}
\frac{	 e^{-i \omega^2 (n^2-j) s} -1 }{- i \omega^2 (n^2-j)}	
\\
		%%%%%%%%%%%%%%%%%%%%%%%%%%%%%%%%
	&= 
	\phi_n e^{i \omega^2 n^2 t}  + 
	\sum_{ 
		n \leq j \leq n^2 - (p-1)(2n-p)
	}
	\frac{	 	(c^p)_{n,j} }{ \omega^2 (n^2-j)}	
	\left(
	e^{i \omega^2 j t} - e^{i \omega^2 n^2 t} 
	\right)
	\label{eq:Bintegral}
	\end{align*} 
The formula in \eqref{eq:FF_Recursive}  and the inclusion $c \in \cX$ follow. 
	
\end{proof}

With this recursive definition, we can compute the Fourier coefficients of a solution to any order.
However it is not clear whether the solution is bounded. 
To show this series  converges, we wish to show that the unique $ c \in \cX$ solving \eqref{eq:FF_Recursive} in fact  has a bounded norm in $ X$. We define an operator $ T: X \to X$ whose fixed points correspond to sequences satisfying the recursive relation. 
We will then show for sufficiently small $\phi$ the operator $T$ has a fixed point. 

 To that end, for fixed $ \omega \in \R^d$  define bounded linear operators $K, L : X \to X $ and $ \iota:  \ell^1_{s,d}  \to X $  by 
	\begin{align*}
		(K c)_{n,j} &= 
		\begin{dcases}
			\frac{c_{n,j}}{\omega^2(n^2-j)} & \mbox{if } p \leq n \mbox{ and }  n \leq j \leq n^2 - (p-1)(2 n -p)  \\
			0 & \mbox{otherwise}
		\end{dcases} \\
	%%%%%%%%%%%%%%%
		(L c)_{n,j} &= 
		\begin{dcases}
			\sum_{ n \leq k < n^2} c_{n,k} & \mbox{if } j = n^2 \\
			0 & \mbox{otherwise}
		\end{dcases}  \\
(\iota \phi)_{n,j} &= 
	\begin{dcases}
	\phi_n & \mbox{if } j = n^2 \\
	0 & \mbox{otherwise.}
	\end{dcases} 
	\end{align*}
Note that $K$ is a compact map, that is it can be well approximated by finite dimensional maps. We may compute  norms 
\begin{align*}
	\|K \| &=   \frac{1}{ \|\omega\|^2  p(p-1)},   
	&
	\|L\| &= 1,  
	&
	\| \iota\| &= 1,
\end{align*} 
where $\|\omega\|^2  = \sum_{1 \leq i \leq d} \omega_i^2$. 
Note that these norms do not depend on $s$, the spatial regularity of our initial data $ \phi \in \ell^1_{s,d}$.  

\begin{definition} \label{def:Newton_Map}
	For fixed  $ \phi \in \ell^1_{s,d}$, define the operator $T  : X \to X$ by
\[
T(c) := \iota \phi +  (I -L)K  c^p  
\]
\end{definition} 
By construction, fixed points $c=T(c) $ will satisfy the recursive relation in \eqref{eq:FF_Recursive}. 
It then follows from Lemma \ref{prop:FF_Recursive} that fixed points of $T$ correspond with (quasi)periodic solutions of \eqref{eq:NLS}. 
To prove Theorem \ref{prop:Intro_Quasiperiodic}  we aim to show that $T$ has a fixed point in a ball 
\[
B_r(\hat{c}) = \{ c \in X : \| c - \hat{c} \| \leq r \} . 
\]  
We note that the mapping $ c \mapsto c^p$ is Fr\'echet differentiable and $ K$ is a compact map, thereby $T$ is also compact and Fr\'echet differentiable.  
By the Schauder fixed point theorem, if there exists some $ \hat{c} \in X$ and $ r>0$  for which  $ T$ maps the ball $B_r(\hat{c})$ inside of itself, then there must exist at least one point  $ \tilde{c} \in B_r(\hat{c})$ for which $ T(\tilde{c}) = \tilde{c}$.  
Moreover this fixed point must be unique, as each of its coefficients are uniquely determined by the recursion relation in \eqref{eq:FF_Recursive}.

\begin{proof}[Proof of Theorem \ref{prop:Intro_Quasiperiodic}]
	For initial data $ u_0 \in A_s^+(\T^d)$ let us denote its Fourier coefficients by $ \hat{u}_0 = \phi \in \ell_{s,d}^1$, and the maximal interval of existence of its solution by $J$. 
	From Corollary \ref{thm:ExplicitSolution} there exists a function $ a : J \to \ell_{s,d}^1$ such that $ u(t,x) = \sum_{n \in \N^d} a_n(t) e^{i \omega m x}$. 
	By Lemma \ref{prop:FF_Recursive}  there exists a unique sequence $c \in \cX$ such that the component functions $a_n(t)$ can be expressed by the Fourier polynomials in \eqref{eq:PeriodicCoefficients} having frequency vector $ \{ \omega_i^2 / 2 \pi \}_{i=1}^d$. 
	To show that $u(t)$ is globally defined and in fact a (quasi)periodic solution, it suffices to prove that $ \| c \|  < \infty$. 
	This is equivalent to showing the operator $T$, given in Definition \ref{def:Newton_Map}, has a fixed point in some ball of finite radius. 
	
	Since $T$ is a compact Fr\'echet differentiable map, by the Schauder fixed point theorem it suffices to find conditions on $r$ and $\phi$ such that $\sup_{ c \in B_r(0)} \| T(c) \| \leq r$.   
	By applying norm estimates on the linear operators and the Banach algebra property of $X$, we obtain  
	\begin{align*}
	\sup_{ c \in B_r(0)} \| T(c) \| &\leq \| \phi \| + \sup_{ c \in B_r(0)} \|   (I-L)K \|  \cdot \| c^p \| \\
%	&\leq \| \phi \| + 	\sup_{ c \in B_r(0)} \frac{2 \| c\|^p}{\|\omega\|ega^2 p (p-1)} \\
	&\leq \| \phi\| + \frac{2 r^p}{\|\omega\|^2 p(p-1)} .
	\end{align*}
	To focus on when this expression is bounded above by $r$, we define\footnote{Note that function $P$ is unrelated and not to be confused with functions $P_n$ in \eqref{eq:ExplicitFormula}.}  
	\[
	P(r) := \| \phi \| + \frac{2 r^p}{\|\omega\|^2 p(p-1)} -r .
	\]

To find a value of $r$ such that $P(r)\leq 0$ we choose the minimum value of $P(r)$ and see where this is non-positive. 	
We may compute   that  $0=P'(r) = \frac{2 r^{p-1}}{\|\omega\|^2 (p-1)} -1$ is solved precisely when $ r_0 = \left( \frac{\|\omega\|^2 (p-1)}{2}\right)^{1/(p-1)} $. 
Plugging this in we obtain 
\[
P(r_0) = 
\|\phi\| - \left(
1-\frac{2}{\|\omega\|^2p(p-1)} \frac{\|\omega\|^2(p-1)}{2}
\right) r_0 
=
\|\phi\| - \frac{p-1}{p} \left( \frac{\|\omega\|^2 (p-1)}{2}\right)^{1/(p-1)}
\]
Hence if $\|\phi\| \leq    \frac{p-1}{p} \left( \frac{\|\omega\|^2 (p-1)}{2}\right)^{1/(p-1)}$ then  $P(r_0) \leq 0$.  
Hence, the ball $B_{r_0}(0)$ is mapped into itself. 
By the Schauder fixed point theorem, there exists a fixed point of the map $T$ in the ball. 
Hence for $c$ defined above, we have $ c = T(c)$ with  $ \| c\| \leq r_0$, and moreover $ \|u(t)\| \leq r_0$ for all $ t \in \R$. 

\end{proof}

%%%%%%%%%%%%%
%%%%%%%%%%%%%
%%%%%%%%%%%%%
%%%%%%%%%%%%%

\section{Solutions with monochromatic initial data}
\label{sec:Monochromatic}
When studying evolutionary PDEs it is natural to begin by studying solutions having a single Fourier mode as their initial data. 
For example, the classical NLS $i u_t = \triangle u \pm |u|^{p-1} u$ has periodic solutions 
\[
 u(t,x) = A e^{it(  (n\omega)^2 \mp |A|^{p-1} )} e^{i n\omega x} , \qquad A \in \C
\]  
In contrast, the only solutions to \eqref{eq:NLS} of the form $ u(t,x) = \psi(t) \phi(x)$ are either constant in $ t$ or constant in $x$, cf \cite{jaquette2020global}. 
Nevertheless, as worked out in Examples \ref{eg:Recursive} and  \ref{eg:ExplicitFunctions}, we can explicitly solve for the Fourier coefficients of a solution with monochromatic initial data. 
Note from Example \ref{eg:ExplicitFunctions}  that solutions $a_n(t)$ are related by a scaling $ A^n / \omega^{2(n-1)}$. 
In Theorem \ref{prop:Rescaling} below we show that if the initial data is a single Fourier mode, then all of the solution coefficients are related by geometric scaling in the spatial modes. 
Throughout  this section, we study \eqref{eq:NLS} with $d=1$ and monochromatic initial data.  
 
\begin{theorem}  \label{prop:Rescaling}

Fix $ p \geq 2$ and $ d =1$ and consider   \eqref{eq:NLS}  	 with initial data $u_0(x) =   A e^{ i \omega  x}$, and 
define $ c(A,\omega) \in \cX$ as the unique sequence  satisfying  the recursive relationship in \eqref{eq:FF_Recursive}. 
If $ \tilde{c} = c(1,1)$ then 
\begin{align} \label{eq:GeometricScaling}
	\big(c(A,\omega)\big)_{n,j} &=  \frac{A^n}{ \omega^{\frac{2(n-1)}{p-1}}  } \tilde{c}_{n,j}  .
\end{align}

\end{theorem}
\begin{proof} 
	We show that if the coefficients $c(A,\omega)$ are defined as in \eqref{eq:GeometricScaling} then they will satisfy the recursive relationship in \eqref{eq:FF_Recursive}. 
When $n=1$, then  \eqref{eq:GeometricScaling}  says that $ c(A,\omega)_{1,j} = A \tilde{c}_{1,j} = A \delta_{1,j}$ using the Kronecker delta,  whereby  \eqref{eq:FF_Recursive} is satisfied. 
For $n \geq 2 $ we first simplify the product $ (c^p)_{n,j}$ using the assumption in \eqref{eq:GeometricScaling}. 
	\begin{align}
		(c^p)_{n,j}
		&= 
		\sum_{\substack{k_1 + \dots + k_p = n \\ 1 \leq k_1, \dots , k_p < n } }
		c_{k_1,j_1} 
		\cdots   
		c_{k_p,j_p}  \nonumber
		\\ 
		 &= 
		\sum_{\substack{k_1 + \dots + k_p = n \\ 1 \leq k_1, \dots , k_p < n } }
		\left(
		\frac{A^{k_1}}{ \omega^{\frac{2({k_1}-1)}{p-1}}  } \tilde{c}_{k_1,j_1} 
		\right)
		\cdots   
		\left(	 	 \frac{A^{k_p}}{ \omega^{\frac{2({k_p}-1)}{p-1}}  } \tilde{c}_{k_p,j_p} 
		\right)  
		\nonumber  \\ 
		&=  \frac{A^{n}}{ \omega^{\frac{2(n-p)}{p-1}}  } (\tilde{c}^p)_{n,j} . \label{eq:SimplifiedRescalingProduct}
	\end{align}

Our  assumption that $ \tilde{c} = c(1,1)$ satisfies \eqref{eq:FF_Recursive} for $n \geq 2$ can be restated as 
	\begin{align}  \label{eq:SimpleRecursionRelation}
	\tilde{c}_{n,j} &= 
	\begin{dcases} 
	\frac{(\tilde{c}^p)_{n,j} }{		(n^2-j) \;\;}
	& \mbox{ if } n\leq j<n^2  \\
 -
	\sum_{n \leq k < n^2} \frac{(\tilde{c}^p)_{n,k} }{		(n^2-k) \;\;}
	& \mbox{ if } j = n^2 .
	\end{dcases}
	\end{align}
	For the components $c(A,\omega)_{n,j}$ with $j < n^2$, we can simplify the RHS of \eqref{eq:FF_Recursive}, using \eqref{eq:SimplifiedRescalingProduct} and \eqref{eq:SimpleRecursionRelation} as below  
	\begin{align} \nonumber 
	\frac{(c^p)_{n,j} }{ \omega^2		(n^2-j) \;\;} &=
	 \frac{A^{n}}{ \omega^{\frac{2(n-p)}{p-1}}  }
	  \frac{(\tilde{c}^p)_{n,j}}{\omega^2(n^2-j)}
	   \\
		& = \frac{A^n}{ \omega^{\frac{2(n-1)}{p-1}}  } (\tilde{c} )_{n,j} %\\
\label{eq:RescalingMatched}
	\end{align} 
	By the rescaling given in \eqref{eq:GeometricScaling}, the result in \eqref{eq:RescalingMatched} equals $ (c)_{n,j}$. Hence  recurrence relation  \eqref{eq:FF_Recursive} is satisfied for $ j<n^2$.  
By the same argument the recurrence relation \eqref{eq:FF_Recursive} for coefficients $ c_{n,j}$ is satisfied for $ j=n^2$.

\end{proof}

Thus 
Theorem \ref{prop:Rescaling} shows that the $n$th coefficient is essentially being scaled by 
 $ \left(  A / \omega^{\frac{2}{p-1}  } \right)^n$.  
 If the ratio $A / \omega^{\frac{2}{p-1}} $ is very small, then we would expect the coefficients of $ c(A,\omega)$ to shrink geometrically, and expect the solution $u(t)$ will correspond to a periodic orbit with period  $2 \pi / \omega^2$. 
 Likewise, if the ratio $A /\omega^{\frac{2}{p-1}} $ is very large, then we would expect the coefficients of $ c(A,\omega)$ to grow geometrically, and we may expect the solution to blowup in finite time. 
 
 Moreover, if we fix some $p\geq 2 $ and $\omega =1$ and vary $A$,  we might expect there to be some value $A^* >0 $ such that if $A < A^*$ then $ \| c(A,1) \|_X < \infty$ and if $ A > A^*$ then $ \| c(A,1) \|_X = \infty$.  
 Due to the geometric scaling of $c(A,1)$, the critical value $A^*$ is not sensitive to the choice of $ \ell^1$ for the norm on $X$. 
 We formalize this notion below. 
 \begin{definition}
 	Fix $ p \geq 2$ and $ d =1$ and consider   \eqref{eq:NLS}  	 with initial data $u_0(x) =   A e^{ i   x}$, and 
 	define $  c(A,1) \in \cX$ as the unique sequence  satisfying  the recursive relationship in \eqref{eq:FF_Recursive}. 
 	For $ 1 \leq q \leq \infty$ define 
 	\[
 	A^* = A^{*,q} = \sup \left\{ A> 0 : \| c(A,1) \|_{\ell^q_{0,d} \otimes \ell^q_{0,d}}  < \infty
	\right\} .
 	\]
 \end{definition}
\begin{proposition}
	The constant $ A^*$ is independant of $q$. 
\end{proposition}
\begin{proof}
	Fix $ 1 \leq q < \infty$. We show that $ A^{*,q} = A^{*,\infty}$. For  $ \tilde{c} = c(1,1) \in \cX$, it follows from Lemma \ref{prop:Rescaling} that 
	\[
	A^{*,\infty} = \sup 
	\left\{ 
	A > 0 : \sup_{n \geq 0} \left( A^n \sup_{n \leq j \leq n^2} |\tilde{c}_{n,j}| \right) < \infty
	\right\} .
	\]
	
	If $ A > A^{*,\infty}$, then there exists a subsequence $\{ \tilde{c}_{n_i,j_i}\}_{i \in \N}$ such that $ | A^{n_i} \tilde{c}_{n_i,j_i} | >1$. 
	Thereby $ \| c(A,1) \|_{\ell^q_{0,d} \otimes \ell^q_{0,d}} = + \infty$, and hence $ A^{*,q}  \leq A^{*,\infty}$.  
	
		Suppose that $ A < A^{*,\infty}$. 
		Fix constants $ A_0 = \frac{A+A^{*,\infty}}{2}$,   $ M = \sup_{n \geq 0} \{ A_0^n \sup_{n \leq j \leq n^2} |\tilde{c}_{n,j}|  \} $, and  $ \epsilon = | A|/A_0 $, whereby $M< \infty$ and $ \epsilon < 1$. 
		Then we may calculate 
		\begin{align*}
			\| c(A,1) \|^q_{\ell^q_{0,d} \otimes \ell^q_{0,d}} &=  \sum_{n \geq 0 } \sum_{ n \leq j \leq n^2} | A^n \tilde{c}_{n,j} |^q \\
			&\leq \sum_{n \geq 0 } n^2
			 \left(
			  \frac{|A|}{A_0}
			\right)^{nq}
			 \left(
			 A_0^n \sup_{n \leq j \leq n^2} |\tilde{c}_{n,j}|
			\right)^q \\
			&\leq \sum_{n \geq 0 } n^2 \eps^{nq} M^{q}
		\end{align*}
		Since $ \eps < 1$ then it follows that the sum above is finite, and hence $ A^{*,q}  \geq A^{*,\infty}$.   Thus $ A^{*,q}  = A^{*,\infty}$.   
	
\end{proof}

For each $ p \geq 2$ we can make an empirical estimate for this value of $ A^*$ by using  the formula in \eqref{eq:FF_Recursive} to algorithmically compute the coefficients to any order, and then observe whether the coefficients appear to be growing or shrinking exponentially. 
For $p=2$ we estimate that $ A^*$ is about  $3.37 $, which was computed using a linear regression to fit $\ln \sum_{j=n}^{n^2} | c(1,1)_{n,j}| $ using $ 100 \leq n \leq 300$ (the code for this calculation is available at \cite{jaquette2020global}). 
 While the statistical significance of this test was satisfactory, $R^2 = 0.9999996$, the residual errors are not normally distributed and the estimated value of $A^*$ is sensitive to the upper and lower limits of $n$ used to fit the data.

 In the remainder of the paper, we prove rigorous bounds on the value of $A^*$ for $ p=2$. 
 To obtain a lower bound, note that from Theorem \ref{prop:Intro_Quasiperiodic}  it follows that   $ A^* \geq   \frac{1}{4}$.  
 Moreover if $ |A| \leq  \frac{1}{4}$ then  $\| c(A,1)\| \leq  \frac{1}{2}$.  
To more closely approach the critical value, we use a computer assisted proof in Theorem \ref{prop:Intro_Blowup} (a) to show that $ A^* \geq 3$. 
We begin first  by proving Theorem \ref{prop:Intro_Blowup} (b), that if $A \geq 6$ then the solution will exhibit finite time blowup in the $L^2$ norm.

 \begin{proof}[Proof of Theorem \ref{prop:Intro_Blowup} (b)]
 	Consider the intial data $ u_0(x) = A e^{i \omega x}$ with $A = 1 = \omega$ and fix $\tilde{c} = c(1,1)   \in \cX$ as the unique solution to the recursion relation in \eqref{eq:FF_Recursive}.  
 	We prove finite time blowup by first  proving a lower bound on the  Fourier coefficients $\tilde{c}_{n,n}$, and then using  Parseval's theorem to obtain a lower bound on the $L^2$ norm.  
 	We then apply the scaling from Theorem \ref{prop:Rescaling} to show that the $\ell^2$ norm of the coefficients $c(A,1)$ is unbounded for $ |A| \geq 6$. 
 	
 	To estimate $ \tilde{c}_{n,n}=c(1,1)_{n,n}$ note first that $ \tilde{c}_{1,1} = 1$. 
 	Since $c_{n,j} = 0$ for all $ j< n$,  the formula in \eqref{eq:FF_Recursive} for $ \tilde{c}_{n,n}$ reduces to  
 	\begin{align} \label{eq:CnnFormula}
 	\tilde{c}_{n,n} = \sum_{k=1}^{n-1} \frac{\tilde{c}_{k,k} \tilde{c}_{n-k,n-k}}{n^2 -n}, 
 	\end{align}
 	for all $ n \geq 2$.
 	Note then that $\tilde{c}_{n,n}$ will  always be a positive real number. 
 	
 	We prove by induction that a lower bound $  \tilde{c}_{k,k}  \geq \alpha \frac{ k }{\beta^k}$ is satisfied for some $ \alpha, \beta >0$ (and in particular for $ \alpha=\beta = 6$). 
 	As $\tilde{c}_{1,1}=1$, this inequality is satisfied in the base case if and only if $\beta \geq \alpha$. 
 	For the inductive step, we obtain 
 	\begin{align*}
 	\tilde{c}_{n,n} &\geq  
 	\frac{1}{n^2-n} 
 	\sum_{k=1}^{n-1} 
 	\left(  
 	\frac{\alpha k}{\beta^k}
 	\right)
 	\left(  
 	\frac{\alpha (n-k)}{\beta^{n-k}}
 	\right) \\
 	%%%%%%%%%
 	& = 
 	\frac{1}{n^2-n}  
 	\frac{\alpha^2}{\beta^n}
 	\sum_{k=1}^{n-1}  kn - k^2 \\
 	%%%%%%%%
 	& = 
 	\frac{1}{n(n-1)}  
 	\frac{\alpha^2}{\beta^n} 
 	\left(
 	n \frac{(n-1)n}{2} - 
 	\frac{(n-1)n(2(n-1)+1)}{6}
 	\right) \\
 	%%%
 	& = 
 	\frac{\alpha^2}{\beta^n} 
 	\left(
 	\frac{n+1}{6}
 	\right) .
 	\end{align*}
 	Hence,  the inductive step is satisfied if 
 	\begin{align*}
 	\frac{\alpha^2}{\beta^n} 
 	\left(
 	\frac{n+1}{6}
 	\right) 
 	&	\geq \frac{\alpha n}{\beta^n} . 
 	\end{align*}
 	After simplifying, we find that we need $\alpha \geq 6$. 
 	Hence, we can take $ \alpha = \beta = 6$, and thereby 
 	\[
 	|\tilde{c}_{n,n}| \geq \frac{6n}{6^n}. 
 	\]
 	
 	To prove finite time blowup in the $L^2$ norm, consider now the solution $ u(t,x)$ through the initial data $u_0(x) = A e^{i \omega x}$ with $|A|/\omega^2 \geq 6$.  
 	Fix $c = c(A,\omega)   \in \cX$ as the unique solution to the recursion relation in \eqref{eq:FF_Recursive}. For $ \tilde{c} = c(1,1)$ as above, it follows from Lemma \ref{prop:Rescaling}   that $ c_{n,j} = 	\frac{A^n}{\omega^{2(n-1)}}
 	\tilde{c}_{n,j}  $. 
 	Let 	$J^{(1)} \subseteq \R$ denote the  maximal time of existence of $u(t)$ in $A^+_0(\T^1)$. 
 	Then for all $ t \in J^{(1)}$ the function $u$ is given by: 
 	\begin{align} \label{eq:BlowupSolution}
 			u(t,x) &=
 		\sum_{n=1}^{\infty} 
a_n(t)
 		e^{i \omega n x} ,
 		%%%%%%%%%%%%%
 		&
 		%%%%%%%%%%%%%
 		a_n(t) &=	\sum_{n \leq j \leq n^2} 
 		\frac{A^n}{\omega^{2(n-1)}}
 		\tilde{c}_{n,j} 
 		e^{i \omega^2 j t} .
 	\end{align}
 	Moreover we have $  J^{(1)} = (-T_{min}^{(1)},T_{max}^{(1)})$ for constants defined as  
 	\begin{align*}
 	-T_{min}^{(1)} &= \inf \left\{ t < 0 : \sum_{n=1}^\infty | a_n(t) | < \infty \right\}
 	,
 	&
 	T_{max}^{(1)} &= \sup \left\{ t > 0 : \sum_{n=1}^\infty | a_n(t) | < \infty \right\}.
 	\end{align*}

 	Note furthermore that the PDE \eqref{eq:NLS} with $(p,d)=(2,1)$ is locally well-posed on $L^2(\T^1)$, and even $ H^{s}(\T^1)$ for $ s > \frac{-1}{2}$, cf \cite{kenig1996quadratic}. Let $ J^{(2)} \supseteq J^{(1)}$ denote the maximal interval of existence of $u(t)$ in $L^2(\T^1)$. 
 	By the construction made in Corollary \ref{thm:ExplicitSolution}, the solution in \eqref{eq:BlowupSolution} satisfies Duhamel's formula for all $ t \in J^{(2)}$. 
 	As Parseval's identity gives a equivalence between the norms of $L^2$ and $ \ell^2$, we may state the maximal interval of existence as 	$  J^{(2)} = (-T_{min}^{(2)},T_{max}^{(2)})$ where  	
 	\begin{align*}
 	-T_{min}^{(2)} &= \inf \left\{ t < 0 : \sum_{n=1}^\infty | a_n(t) |^2 < \infty \right\}
 	,
 	&
 	T_{max}^{(2)} &= \sup \left\{ t > 0 : \sum_{n=1}^\infty | a_n(t) |^2 < \infty \right\}.
 	\end{align*}

 We now argue by contradiction to prove finite time blowup; suppose that $ T_{max}^{(2)}  > \frac{2 \pi }{\omega^2}$.  
 Thereby  
 $u \in 
  L^2\big([0,\frac{2 \pi }{\omega^2}] , L^2(  \T , \C)\big)
  \cong 
  L^2\big([0,\frac{2 \pi }{\omega^2}] \times  \T , \C\big) $, and by the formula in \eqref{eq:BlowupSolution} the solution $u$ is $2 \pi / \omega$ periodic in space and $2 \pi / \omega^2$ periodic in time.  
By Parseval's theorem it follows that  
 	\begin{align*}
 		\int_{[0,\frac{2\pi}{\omega^2}]\times \T} |u(t,x)|^2& = 
 		\left(
 		\frac{2 \pi}{\omega^2}
 		\right) 
 		\left(
 		\frac{2 \pi}{\omega}
 		\right) 
 		\sum_{n=1}^{\infty} 
 		\sum_{n \leq j \leq n^2} 
 		\left|
 		\frac{A^n}{\omega^{2(n-1)}}
 		\tilde{c}_{n,j} 
 		\right|^2  \\
 		%%%%%%%%%%%%%%%
 		&>  
 		\frac{4 \pi^2}{\omega^3} 
 		\sum_{n=1}^{\infty}  
 		\left|
 		\frac{A^n}{\omega^{2(n-1)}}
 		\tilde{c}_{n,n} 
 		\right|^2  \\ 		
 		%%%%%%%%%%%%%%%
 		&\geq
 		4\pi^2 \omega 
 		\sum_{n=1}^\infty 
 		\left|
 		\frac{A^n 6n}{\omega^{2n} 6^n}
 		\right|^2
 	\end{align*}
% 	where the inequality follows from only focusing on the $\tilde{c}_{n,n}$ terms. 
 	If $A \geq  6\omega^2$ then the series diverges, whereby $\| u \|_{L^2([0,\frac{2 \pi }{\omega^2}] \times \T )} = + \infty$. 
 	This  contradiction proves that the maximal time of existence is bounded as  $  T_{max}^{(2)}  \leq   \frac{2\pi}{\omega^2}$.  
 	Since \eqref{eq:NLS} is locally well posed in $L^2$ then it is necessarily the case that $\limsup_{t \to  T_{max}^{(2)} } \| u(t)\|_{L^2} = + \infty$.
 \end{proof}

This blowup result trivially extends to the quadratic case of \eqref{eq:NLS} posed on $ \T^d$ for $ d \geq 2$.  
We also obtain the following extension. 
 
 \begin{corollary} \label{prop:GeneralBlowup}
 	Consider     \eqref{eq:NLS}  with $(p,d)=(2,1)$ and  	   initial data $u_0(x) =   A e^{ i \omega  x} + \sum_{n \geq 2} \phi_n e^{i \omega n x}$ for $A \in \C$ and $ \omega >0$.   
 	If $ A/\omega^2 \geq  6$ then the solution will   blowup in the $L^2$ norm in finite time. 
 \end{corollary}
 \begin{proof}
 	The formula in $ \tilde{c}_{n,n}$ in \eqref{eq:CnnFormula} stays the same, so the rest of the proof is identical to that of Theorem \ref{prop:Intro_Blowup} (b).
 \end{proof}

  From our numerical calculation of $c(1,1)$ with $p=2$ we estimate that the critical value $A^*$ is approximately $3.37$.  
  Using a computer assisted proof, we are able to rigorously show that if $|A | \leq 3$ then $ \| c(A,1) \| < \infty$. 
  The approach of this proof is in the same spirit as the proof in  Theorem \ref{prop:Intro_Quasiperiodic}, where we showed that $T$ has a fixed point in some ball $B_r(\hat{c})$.  
In Theorem \ref{prop:Intro_Quasiperiodic} we considered arbitrary initial data and applied the Schauder fixed point theorem to a  ball centered about $ \hat{c} =0$.
If we have explicitly given initial data, then we can explicitly compute Fourier coefficients up to any fixed order.  We can obtain sharper results by centering about this finite truncation.

  We note that all of these coefficients $c(3,1)_{n,j}$ are rational and it would be possible to exactly represent these coefficients with integers on a computer. Unfortunately the size of the denominators grows quite rapidly. To save computational expense, we represent each coefficient as an interval which contains the precise value and whose endpoints are floating point numbers. By using interval arithmetic we are able to rigorously manipulate these coefficients and keep track of any rounding error (cf \cite{gomez2019computer}).

  To show that $T$ maps a ball $B_r(\hat{c})$ into itself and moreover is a contraction, we apply a  parameterized Newton-Kantorovich theorem. 
  The two parameters in the theorem are $\hat{c}$, the approximate fixed point,   and $r$,  the radius of the ball.   
  Such a theorem is commonly used in computer assisted proofs of nonlinear dynamics, see for example  \cite{jaquette2020global} studying \eqref{eq:NLS} and more generally \cite{van2015rigorous,james2017validated,gomez2019computer}. 
  This version of the Newton-Kantorovich theorem reduces the question of whether $T$ maps a ball of radius $r$ into itself, down to whether an explicitly computable polynomial $P(r)$ satisfies a single inequality, see   \eqref{eq:RadiiPolynomial}, and is sometimes referred to as the radii polynomial approach.

  \begin{theorem}[cf \cite{james2017validated}] % Proposition 2.6 
  	\label{prop:RadiiPolynomial}
  	Let $X$ be a Banach space, $ \hat{c} \in X$, and suppose  that $T: X \to X$ is a Fr\'echet differentiable mapping.
  	Suppose that $ Y_0, Z_1$ are positive constants and   $ Z_2: (0,\infty) \to [0, \infty)$ is a positive function, having   
  	\begin{align}\label{eq:Y0_thm}
  	\|T(\hat{c}) - \hat{c} \|_X &\leq Y_0 ,
  	\end{align}
  	and that 
  	\begin{align}\label{eq:Zbound_thm}
  	\sup_{c \in \overline{B_r( 0)}} \| DT( \hat{c}+c)\|  &\leq Z_1 +  Z_2(r) r 
  	\end{align}
  	for all $ r > 0$. 
  	If there is an $r>0$ so that 
  	\begin{align}\label{eq:RadiiPolynomial}
  	\underbrace{Z_2(r)r^2-(1-Z_1)r + Y_0}_{=:P(r)} \leq 0  
  	\end{align}
  	then there exists a unique $ \tilde{c} \in \overline{B_r(\hat{c})}$ such  that $ T(\tilde{c}) = \tilde{c}$. 
  \end{theorem}

  For any fixed values $A \in \C$ and  $ N \in \N$,  we can explicitly compute the coefficients of $ c(A,1)_{n,j}$  for all $ 1 \leq n \leq N$ and $ n \leq j \leq n^2$. 
  Thus we do not need to worry about $T$ being a contraction on the first finitely many coefficients. 
  More formally, let us define subspaces $ X = X_N \oplus X_\infty$ below
\begin{align*}
	X_N &:= \{
	c \in X : c_{n,j} \neq  0 \mbox{ only if } n \leq N
	\} \\
	X_\infty &:= \{
	c \in X : c_{n,j} \neq 0 \mbox{ only if } n >  N
	\} 
\end{align*}  
and define  projection operators $ \pi_N: X \to X_N $ and $ \pi_\infty : X \to X_\infty$.

For fixed initial data $ u_0(x) = Ae^{i x}$  define $ \hat{c} \in X_N$ such that $ \hat{c} = \pi_N c(A,1)$ where $ c(A,1) \in \cX$ is the unique sequence satisfying \eqref{eq:FF_Recursive}. 
Hence $ \hat{c} = \pi_N ( \hat{c})$ and furthermore $ \hat{c} = \pi_N  T(\hat{c})$. 
Using interval arithmetic to evaluate \eqref{eq:FF_Recursive} we can obtain a rigorous enclosure of  every component of $\hat{c}$. 
If $ \| c(A,1)\| < \infty$, then there exists some $ c_\infty \in X_ \infty$ such that $ c(A,1)  = \hat{c}+ c_\infty$. 
Moreover $c_\infty$  will be a fixed point of the map 
$
T_{\infty} : X_\infty \to X_\infty 
$  
defined for $u \in X_\infty$ by 
\begin{align}\label{eq:TinftyDef}
	T_\infty (u) &= \pi_\infty T(\cTrunc +u) .
\end{align}
Thus for an explicitly computed finite sequence $ \hat{c} \in X_N$, finding a fixed point of $T$ is equivalent to finding a fixed point of $ T_\infty$. To do so, we will aim to apply Theorem \ref{prop:RadiiPolynomial} for the map $T_\infty$ in a ball $ B_{r}^{\infty}(0) = \{ c \in X_{\infty} : \| c \| \leq r \}$.

We first prove a lemma which defines $Y_0,Z_1,Z_2$ and shows that they satisfy the inequalities in \eqref{eq:Y0_thm} and \eqref{eq:Zbound_thm}.
We also note that as $\pi_\infty T( \cTrunc )    \in X_{2N}$, one could obtain  the sharpest $Y_0$ bound by computing an interval enclosure of $\|T_\infty(0) \| $ using finitely many arithmetic operations. 
However  this can be computationally expensive if $N$ is large.  
We instead use a faster, rougher bound as given in \eqref{eq:YboundCAP}.  
This is no great loss, as the $Z_1$ bound is the most difficult one to control in practice. 

\begin{lemma} \label{prop:RadiiPolynomial_CAP} 
	Fix $(p,d)=(2,1)$, fix   $ \cTrunc \in X_N$ and define $ T_\infty$ as in \eqref{eq:TinftyDef}. 
	Define $b \in \ell^1_{0,1} $ by $	b_{n} = \sum_{n \leq j \leq n^2} |\hat{c}_{n,j}| $ and
	 $Y_0,Z_1,Z_2(r)$ by
	\begin{align} \label{eq:YboundCAP}
		Y_0 &= \frac{1}{\omega^2}
		\sum_{N+1 \leq n \leq 2 N } \frac{(b * b)_n}{n-1}, 
		\\
		Z_1 &= \frac{4}{\omega^2}
		\sum_{1\leq n \leq N } 
		\sum_{n\leq j \leq n^2}  \frac{|\cTrunc_{n,j}|}{n^2 +2 n (N+1)  - j} ,
				\label{eq:Z1bound}
				 \\
		Z_2(r) &=  	\frac{2 r}{\omega^2(N+1)^2} ,
		\label{eq:Z2bound}
	\end{align}
	then the following inequalities are satisfied
	\begin{align*}
	\|T_\infty(0) - 0 \|_{X_\infty} &\leq Y_0 ,
	\\
	%%%%%%%%%%%%%
\| DT_\infty(0)\|_{B(X_\infty)}  &\leq Z_1 , \\
%%%%%%%%%%%%%%%%%%%%
	\sup_{c \in \overline{B_r^\infty( 0)}} \| DT_\infty(c) - DT_\infty(0) \|_{B(X_\infty)}  &\leq Z_2(r) r  ,  \qquad \forall r >0.
	\end{align*}  

\end{lemma}

\begin{proof}

We divide the proof into three parts. 

\textbf{The $Y_0$ Bound.} 
First note that 
	\begin{align}
		\| T_\infty(0) \| &= \| \pi_\infty  (I- L) K \hat{c}^2 \|  \nonumber \\
		&\leq 2 \| \pi_\infty K \hat{c}^2 \| . \label{eq:PreY0}
	\end{align}
	 Using the bound $ | (Kc)_{n,j} | \leq \frac{1}{2\omega^2(n-1)} | c_{n,j}|$ for all $ c \in X$ we obtain 
	 \[
\|	 \pi_{\infty} K \hat{c}^2 \| \leq 
\sum_{N+1 \leq n \leq 2 N } \frac{1}{2n-2} \sum_{n \leq j \leq n^2}
 \left|
\sum_{\substack{n_1 + n_2 = n \\j_1 +j_2 = j}}\hat{c}_{n_1,j_2} \hat{c}_{n_2,j_2} 
\right| .
	 \]
	 Define a sequence $ b \in \ell^1_{0,1} $ by $ b_{n} = \sum_{n \leq j \leq n^2} |\hat{c}_{n,j}| $. 
	 By the Banach algebra property of $\ell^1_{0,1}$ we have 
	 \begin{align*}
\|	 \pi_{\infty} K \hat{c}^2\| &\leq 
\sum_{N+1 \leq n \leq 2 N } \frac{1}{2\omega^2(n-1)} \sum_{n_1 + n_2 = n }
\sum_{\substack{j_1 +j_2 = j \\n \leq j \leq n^2}}
\left| \hat{c}_{n_1,j_2} \hat{c}_{n_2,j_2} 
\right| \\
&\leq 
\sum_{N+1 \leq n \leq 2 N } \frac{1}{2\omega^2(n-1)} \sum_{n_1 + n_2 = n } b_{n_1} b_{n_2} \\
& = 
\sum_{N+1 \leq n \leq 2 N } \frac{(b^2)_n}{2\omega^2(n-1)} .
	 \end{align*}
	 Combining this with \eqref{eq:PreY0} we see that $Y_0$ defined in \eqref{eq:YboundCAP} satisfies  $ \| T_\infty(0) \| \leq Y_0$. 
	 \newline

	To obtain the $Z_1$ and $Z_2$  bounds,  we first compute the action of the  Fr\'echet derivative $DT_\infty$  on a vector $ h \in X_\infty$ below
\begin{align*}
	DT_\infty(u)h 
	&= 
	\pi_{\infty } DT(\cTrunc+u) \cdot h \\
	&= 2  (I-L)  \pi_{\infty }  K (\cTrunc + u) * h .
\end{align*}
Using $ \| I - L \| \leq 2$ then  
\begin{align}\label{eq:ZboundStart}
	\sup_{u \in B^\infty_r(0)} 
\|		DT_\infty(u) \| 	&\leq 
4 
\left(
\sup_{ h \in X_\infty, \|h\|=1} 
\| \pi_{\infty }  K \cTrunc *h  \|
+ 
\sup_{\substack{ h \in X_\infty, \|h\|=1 \\ u \in B^\infty_r(0) }}
\| \pi_{\infty }  K u *h \|
\right) .
\end{align}

\textbf{The $Z_1$ Bound.} 
To bound these terms we first focus on $	K \cTrunc * h $, for which the absolute value of its components may be bounded as  
\begin{align}
\left|	(K \cTrunc * h)_{n,j} \right|  &\leq
\frac{1}{\omega^2}
\sum_{\substack{n_1+n_2 =n\\j_1+j_2=j}}  \frac{1}{(n_1+ n_2)^2 - (j_1+j_2)} |\cTrunc_{n_1,j_1} h_{n_2,j_2}|  
\nonumber
\\
%%%%%%%%%%%%%%%%%%%%%
&\leq 
\frac{1}{\omega^2}
\sum_{\substack{n_1+n_2 =n\\j_1+j_2=j}}  \frac{1}{n_1^2 +2 n_1 n_2  - j_1} |\cTrunc_{n_1,j_1} h_{n_2,j_2}|
\nonumber
   \\
%%%%%%%%%%%%%%%%%%%%%
&\leq 
\frac{1}{\omega^2}
\sum_{\substack{n_1+n_2 =n\\j_1+j_2=j}} 
\left|
\frac{\cTrunc_{n_1,j_1} }{n_1^2 +2 n_1 (N+1)  - j_1}
\right| |h_{n_2,j_2}|	  .
\label{eq:Z1ComponentBound}
\end{align}
Here we first used the fact that   $  j_2 \leq n_2^2$, and   then \eqref{eq:Z1ComponentBound} follows from  the fact that  $ h \in X_\infty$ so $h_{n,j}=0$ for $n \leq N$. 
Setting aside $ 1/\omega^2$ for a moment, the terms in \eqref{eq:Z1ComponentBound} may be given by the Cauchy product of the sequence  $\{ |h_{n,j}|\} \in X_\infty$ and the sequence $\left\{ \left|
\frac{\cTrunc_{n,j} }{n^2 +2 n (N+1)  - j}
\right| \right\} \in X_N$. 
Using the Banach algebra property of $ X = X_N \oplus X_{\infty}$, it follows that  
\begin{align}\label{eq:Z1}
	\sup_{h \in X_\infty; \|h \| \leq 1} 
	\| K \cTrunc * h \| 
	&\leq 
	\frac{1}{\omega^2}
	\sum_{1\leq n \leq N } 
	\sum_{n\leq j \leq n^2}
	\left|
	\frac{\cTrunc_{n,j} }{n^2 +2 n (N+1)  - j}
	\right|
	.
\end{align}
Hence, by combining \eqref{eq:Z1} into  \eqref{eq:ZboundStart} we see that $Z_1$ defined in \eqref{eq:Z1bound} satisfies $
\| DT_\infty(\hat{c})\|_{B(X_\infty)}  \leq Z_1$.

\textbf{The $Z_2$ Bound.} 
To bound the $Z_2$ term, we now focus on $	K u * h $,  for which the absolute value of its components may be bounded as  
\begin{align}
\left|	(K u * h)_{n,j} \right|   &=
	\frac{1}{\omega^2}
	\sum_{\substack{n_1+n_2 =n\\j_1+j_2=j}}  \frac{1}{(n_1+ n_2)^2 - (j_1+j_2)} |u_{n_1,j_1} h_{n_2,j_2}| 
	\nonumber
	\\
	%%%%%%%%%%%%%
		&\leq 
			\frac{1}{\omega^2}
			\sum_{\substack{n_1+n_2 =n\\j_1+j_2=j}}  \frac{1}{ 2 n_1 n_2  } |u_{n_1,j_1} h_{n_2,j_2}| 
			\nonumber
	\\
	%%%%%%%%%%%%%
		&\leq 
			\frac{1}{\omega^2}
	\sum_{\substack{n_1+n_2 =n\\j_1+j_2=j}} 
	\left|
	\frac{u_{n_1,j_1} }{2 (N+1)^2 }
	\right| |h_{n_2,j_2}| .
	%%%%%%%%%%%%%
	\label{eq:Z2ComponentBound}
\end{align}
 In the first inequality we  used the fact that that both $    j_1 \leq n_1^2$ and $    j_2 \leq n_2^2$. 
 For the second inequality we used the fact that as    both  $u, h \in X_\infty$ then  $| u_{n_1,j_1} h_{n_2,j_2}| \neq 0$ only if $ n_1,n_2 \geq N+1$. 
Setting aside $ 1/\omega^2$ for a moment, the terms in \eqref{eq:Z2ComponentBound} may be given by the Cauchy product of the sequence  $\{ |h_{n,j}|\} \in X_\infty$ and the sequence $\left\{ 	\left|
 \frac{u_{n,j} }{2 (N+1)^2 }
 \right|  \right\}\in X_\infty$. 
Using $\| u \| \leq r $ and  the Banach algebra property, it follows that  
\begin{align}\label{eq:Z2}
\sup_{u \in B^\infty_r(0)}
	\sup_{ h \in X_\infty, \|h\|=1 } 
	\| K u * h \| 
	&\leq \frac{r}{2 \omega^2 (N+1)^2}. 
\end{align}
Hence, by combining  \eqref{eq:Z2} into  \eqref{eq:ZboundStart} we see that $Z_2(r)$ defined in \eqref{eq:Z2bound} satisfies  
\[
\sup_{u \in B^\infty_r(0)}
\sup_{ h \in X_\infty, \|h\|=1 } 
\|	DT_\infty(\cTrunc + u) \cdot  h \|  \leq  Z_1 + Z_2(r) .
\]

\end{proof}

We now present the computer assisted proof of Theorem \ref{prop:Intro_Blowup} (a). 
	The source code for this  computer assisted proof is available at \cite{bib:codesIntegrableNonconservativeNLS} and  uses the interval arithmetic library INTLAB version 10.1 \cite{Ru99a}. 
The computation was run using  
MATLAB version  2020b and took $60.7$ minutes on a Intel i7-8750H processor.

\begin{proof}[Proof of Theorem \ref{prop:Intro_Blowup} (a)]	
	First we consider the initial data $u_0(x) =   A e^{ i \omega  x}$  for $ A=3$ and $ \omega =1 $. 
	Fix $  c(3,1) \in \cX$, the solution to the recursion relation in \eqref{eq:FF_Recursive}. 
	For fixed $N = 110 $ we  use interval arithmetic to  compute a rigorous enclosure of  $\cTrunc = \pi_N c(3,1)$. 
	To prove $\|c(3,1)\| < \infty$ it suffices to show that $ T_\infty$, as defined in \eqref{eq:TinftyDef}, has a fixed point in some ball $ B_r^\infty(0)$. 
	
	We prove that $ T_\infty$ has a fixed point by way of Theorem  \ref{prop:RadiiPolynomial}. 
	Using interval arithmetic we compute $Y_0, Z_1, Z_2(r)$ as defined in \eqref{eq:YboundCAP}-\eqref{eq:Z2bound},  
	which by Lemma \ref{prop:RadiiPolynomial_CAP}, 
	satisfy the inequalities \eqref{eq:Y0_thm} and \eqref{eq:Zbound_thm}. 	
	For $P(r)$ defined in \eqref{eq:RadiiPolynomial} and fixed $r_0 = 32$, we use interval arithmetic to show that $P(r_0) <  0$. 
	Hence by 	Theorem  \ref{prop:RadiiPolynomial} it follows that there exists a unique $ c_\infty \in \overline{B_r^\infty( 0)}$ such  that $ T_\infty(c_\infty) = c_\infty$. 
	Thereby  $ c(3,1) =  \hat{c} + c_\infty $  and furthermore $ \| c(3,1) \|_X < \infty$.

	 More generally, for any $ A \in \C$ and $\omega>0$ fix $ c(A,\omega) \in \cX$ as the unique solution to the recursion relation in \eqref{eq:FF_Recursive} for the initial data $ u_0(x) = A e^{i \omega x}$. 
	From Theorem \ref{prop:Rescaling} we have 
	$
	c(A,\omega )_{n,j} = \omega^{-2} \left(
	\frac{A}{\omega^2}
	\right)^n
	\big( c(1,1) \big)_{n,j}
	$. 
	The argument above proved that $c(3,1) = T( c(3,1) )$ and  $ \| c(3,1) \|_X < \infty$. 
	This latter statement  may be restated as follows
	\[
	\| c(3,1) \|_{X} 
	=
	\sum_{1\leq n \leq \infty} \sum_{n \leq j \leq n^2} |c(3,1)_{n,j}|
		=
	\sum_{1\leq n \leq \infty} \sum_{n \leq j \leq n^2} 
	1^{-2} \left(
	\tfrac{3}{1^2}
	\right)^n
	\left|	\big( c(1,1) \big)_{n,j} \right|
	< + \infty .
	\]
	It follows that if $ |A|/\omega^2 \leq 3$, then $\| c(A,\omega) \|_{X} \leq \omega^{-2} \| c(3,1) \|_{X} < \infty$.

\end{proof}

By computing more coefficients $\{c_{n,j}\}$ for $1 \leq n \leq N$ and $n \leq j \leq n^2$, then one could more closely approximate the critical value $A^* \approx 3.37$ from below.  
However the memory requirements of the present computer assisted proof are of order $ \cO(N^3)$, and taking a large value of $N$ is computationally expensive. 
On the other hand,  there is ample room for improvement for reducing the bound $ A^* \leq 6$ from Theorem \ref{prop:Intro_Blowup} (b).   
Similar questions on the (quasi)periodicity/blowup of monochromatic initial data may also be asked of \eqref{eq:NLS} for other values of $ p \geq 3$ and $ d \geq 2$.

\section*{Acknowledgments} 
The author would like to thank  M. Beck, A. Delshams,  J.P. Lessard, E. Miranda, A. Takayasu,   and C.E. Wayne for informative discussions.

\bibliography{Bib_NLS}
\bibliographystyle{alpha}

\end{document}